\documentclass[12pt]{article}
\usepackage[T1]{fontenc}

\usepackage{amssymb, amscd, amsmath, amsfonts, amsthm}
\usepackage{latexsym}
\usepackage{extarrows}

\usepackage{graphicx,subfigure,url}
\usepackage{epstopdf}

\usepackage{color}
\usepackage{tabu}

\usepackage{tikz}

\usepackage{mathtools}
\usepackage{enumitem}
\usepackage{pifont}
\usepackage{bbm}
\usepackage{hyperref}
\usepackage{lineno}

\newcommand{\overbar}[1]{\mkern 1.5mu\overline{\mkern-1.5mu#1\mkern-1.5mu}\mkern 1.5mu}

\numberwithin{equation}{section}

\newtheorem{problem}[equation]{Problem}
\newtheorem{theo}[equation]{Theorem}
\newtheorem{rem}[equation]{Remark}

\newtheorem{defin}[equation]{Definition}
\newtheorem{prop}[equation]{Proposition}
\newtheorem{cor}[equation]{Corollary}
\newtheorem{lema}[equation]{Lemma}

\usepackage{xcolor}\definecolor{myblue}{RGB}{0,0,255}

\def\conv{\mathop{\rm conv}\nolimits}

\makeatletter

\newenvironment{figurehere}
  {\def\@captype{figure}}
  {}
\makeatother

\begin{document}

\date{September 20, 2023}
\title{Polytopality of simple games} 
\author{
Marinko Timotijevi\'{c}\\ {\small Faculty of Science}\\[-2mm] {\small University of Kragujevac}
\and
Rade T. \v Zivaljevi\' c\\ {\small Mathematical Institute}\\[-2mm] {\small SASA, Belgrade}
\and
Filip D. Jevti\'{c} \\{\small Mathematical Institute}\\[-2mm] {\small SASA,  Belgrade}
}
\maketitle

\begin{abstract}
The Bier sphere $Bier(\mathcal{G}) = Bier(K) := K\ast_\Delta K^\circ$ and the \emph{canonical fan} $Fan(\Gamma) = Fan(K)$ are combinatorial/geometric companions of a \emph{simple game} $\mathcal{G} = (P,\Gamma)$ (equivalently the associated simplicial complex $K$), where $P$ is the set of players, $\Gamma\subseteq 2^P$ is the set of wining coalitions, and $K:= 2^P\setminus \Gamma$ is the simplicial complex of losing coalitions.  We propose and study a general ``Steinitz problem'' for simple games as the problem of characterizing which games
$\mathcal{G}$  are \emph{polytopal (canonically polytopal)} in the sense that the corresponding Bier sphere $Bier(\mathcal{G})$ (fan $Fan(\Gamma)$) can be realized as the boundary sphere (normal fan) of a convex polytope.

We characterize \emph{(roughly) weighted majority games} as the games $\Gamma$ which are \emph{ canonically (pseudo) polytopal} (Theorems \ref{thm:canonical} and \ref{thm:rough-canonical})  and show, by an  experimental/theoretical argument (Theorem \ref{thm:Bier-10-vertices}), that simple games such that $Bier(\mathcal{G})$ is nonpolytopal do not exist in dimension 3. This should be compared to the fact that asymptotically almost all simple games are nonpolytopal and a challenging open problem is to find a nonpolytopal simple game with the smallest number of players.

\end{abstract}

\section{Introduction}

\emph{Simple games} (von Neumann and Morgenstern 1944; Shapley 1962; Taylor and Zwicker 1999, etc.) are mathematical objects originally used in mathematical economics (cooperative game theory) to model the distribution of power among coalitions of players.

An important class of simple games are the \emph{weighted majority games}, where each player $i\in P$ is associated a \emph{weight} $w_i \in \mathbb{R}^+$ and the \emph{winning coalitions} are sets $A\subseteq P$ whose total weight is above a certain threshold $q$, prescribed in advance.

From a different point of view simple games $\Gamma\subseteq 2^P$ are nothing but \emph{simplicial complexes} in disguise, where $K := 2^P\setminus \Gamma$ is the collection of all losing  coalitions. In this context weighted majority games correspond to \emph{threshold simplicial complexes} $Tr_{w < q}$.

\medskip
The monographs Taylor and Zwicker 1999 (simple games) and Matou\v sek 2003 and 2008 (topological combinatorics of hypergraphs and simplicial complexes), well illustrate how similar mathematical objects can be analyzed and studied from a completely different point of view. On the other hand the dichotomy between \emph{simple games} and \emph{simplicial complexes} provides a wonderful opportunity for interaction and transfer of ideas and methods from one field to another.

\subsection{Bier spheres, canonical fans and simple games}

Each simplicial complex $K\subsetneq 2^{[n]}$ is associated an $(n-2)$-dimensional, combinatorial sphere on (at most) $2n$-vertices, called \emph{Bier sphere} $Bier(K)$ (named after T. Bier, see \cite[Section 5.6]{matousek_using_2008}).

Combinatorics and geometry of Bier spheres and their generalizations have been studied in numerous publications \cite{bjorner_bier_2004, cukic-blowups-2007, delucchi_fundamental_2020, jevtic_bier_2022, jtz-bier-2019, zivaljevic_unavoidable_2020, jojic_tverberg_2021, longueville_bier_2004, timotijevic_dual_2019}. One of the fundamental results (see \cite[Theorem 3.1]{jtz-bier-2019} and \cite[Theorem 2]{jevtic_bier_2022}) says that each Bier sphere $Bier(K)$ admits a \emph{canonical starshaped embedding} such that the corresponding (canonical) radial fan $Fan(K)$ is a coarsening of the \emph{braid arrangement fan}.

\medskip
In turn each simple game $\Gamma \subsetneq 2^{[n]}$ is also associated a Bier sphere $Bier(\Gamma)= Bier(K)$ (a combinatorial object) and the canonical fan $Fan(K)$ (a geometric object), and our general objective is to relate the combinatorics of simple games to the geometry of $Bier(\Gamma)$ and $Fan(\Gamma)$.

\medskip
It is known \cite[Section 5.6]{matousek_using_2008} that asymptotically, the number of Bier spheres on $2n$ vertices is doubly exponential  in $n$, while only an exponential number of them are polytopal. It is known that nonpolytopal spheres begin to emerge already in dimension $3$ and the two nonpolytopal spheres in this dimension (each with the minimum number of eight vertices) are the Barnette and the Gr\" unbaum-Br\" uckner sphere \cite{ewald_geometry_1996}.

\medskip
There are many interesting classes (hierarchies) of simple games (\cite{taylor_games_2000} \cite{gvozdeva_games_2011}), reflecting their increasing complexity. It is well established that the weighted majority games are the simplest possible and the complexity of other games are often measured by how far are they from being a weighted game.

\medskip
Our conjecture  (partially supported by our earlier results about Bier spheres of threshold complexes) is that there exist a correlation between increasingly (non)weighted simple games and increasingly (non)polytopal Bier spheres. More explicitly, we propose the following working hypothesis as a guide for tests and experiments.

\medskip\noindent
{\bf Guiding principle and the working hypothesis:} There is an interesting analogy between (non)weighted simple games and (non)polytopal Bier spheres (canonical fans). The degree of nonweightedness, measured by the complexity of the corresponding (non)trading certificate (\cite{taylor_games_2000} \cite{gvozdeva_games_2011}) should manifest itself on the level of Bier spheres as a \emph{degree of nonpolytopality}. In particular we seek nonpolytopal Bier spheres among simple games of high  complexity of their non-trading certificates.

\medskip
The following theorems are central theoretical results of the paper.

\begin{theo}\label{thm:canonical}
Suppose that $K\subsetneq 2^{[n]}$ is a proper simplicial complex such that $Vert(K)=[n]$. Then $K$ is a threshold complex (equivalently $\Gamma = 2^{[n]}\setminus K$ is a weighted majority game with all weights strictly positive) if and only if the canonical fan $Fan(K)$ of $K$ is polytopal.
\end{theo}

\begin{theo}\label{thm:rough-canonical}
 Let $K\subsetneq 2^{[n]}$ be a proper simplicial complex such that $Vert(K)=[n]$. Then $\Gamma = 2^{[n]}\setminus K$ is a \emph{roughly weighted simple game} with all weights strictly positive if and only if the canonical fan $Fan(\Gamma)$ of $\Gamma$ is \emph{pseudo-polytopal} in the sense that it refines the normal fan of a convex polytope.
\end{theo}

\medskip
According to \cite{lutz_combinatorial_2007} there are 247882 combinatorial spheres with 10 vertices and the problem of deciding which of them are (non)polytopal (\cite[Problem 4]{lutz_combinatorial_2007}) is apparently still wide open.

\medskip
Recall that all simplicial 3-spheres with up to 7 vertices are polytopal, and only two
3-spheres with 8 vertices are nonpolytopal, the Gr\" unbaum-Sreedharan sphere and
the Barnette sphere. The classification of triangulated 3-spheres with 9 vertices into
polytopal and nonpolytopal spheres was started by Altshuler and Steinberg and completed by Altshuler, Bokowski, and Steinberg, see \cite{lutz_combinatorial_2007} for the references.

\medskip
The following theorem is our main experimental  result of the paper.

\begin{theo}\label{thm:Bier-10-vertices}
All Bier spheres with up to eleven vertices are polytopal, in particular this holds for all $3$-dimensional Bier spheres. For illustration, there are $88$ non-threshold complexes on $5$ vertices and $48$ corresponding non-isomorphic Bier spheres.  Explicit convex realizations of all  spheres with 10 and 11 vertices can be respectively found in \cite{kv5d3} and \cite{kv5d4}.
\end{theo}

One of the central problems that still remains to be solved is the following.

\begin{problem}\label{prob:unsolved}
  Find a simple game $\Gamma$ (simplicial complex $K$) with the smallest number of players (vertices) such that the corresponding Bier sphere $Bier(\Gamma) = Bier(K)$ is nonpolytopal.
\end{problem}

\subsection{Glossary}\label{sec:glossary}

The terminology used in the paper is quite standard and self-explanatory. For the readers convenience we include a short glossary of  main definitions and a list of key objects and fundamental constructions.

\begin{itemize}
  \item A \emph{simple game} is a family $\Gamma\subseteq 2^P$ such that $P\in \Gamma, \emptyset\notin \Gamma$ and $A\supseteq B\in\Gamma \Rightarrow  A\in\Gamma$.
    \item A simplicial complex $K\subseteq 2^V$ on a set of vertices $V$ is a \emph{down-set} ($A\subset B\in K \Rightarrow A\in K$) such that $\emptyset\in K$.
  \item $K$ is \emph{proper} if $K\subsetneq 2^V$, or in other words if $\Gamma := 2^P\setminus K$ is a simple game.
  \item   $v\in V$ is called a \emph{ghost vertex (a passer)} if $\{v\}\notin K$ ($\Leftrightarrow \{v\}\in\Gamma$).
  \item Alexander dual of $K$ is the complex $K^\circ := \{A\subseteq [n] \mid [n]\setminus A\notin K\}$.
  \item The Bier sphere associated to $K$ is the \emph{deleted join}
$Bier(K):= K \ast_\Delta K^\circ.$
\end{itemize}

 The vertices of $Bier(K):= K \ast_\Delta K^\circ$ are $[n]\cup [\bar{n}] = \{1,\dots, n, \bar{1},\dots, \bar{n}\}$ and a simplex $\tau \in Bier(K)$ is described as the union $\tau = A_1\cup \overline{A_2}$, where $A_1$ and $A_2$ are subsets of $[n]$ such that $A_1\in K, A_2\in K^\circ, A_1\cap A_2=\emptyset$ and (by definition) $\overline{C}:= \{\bar{i} \mid i\in C\}\subseteq [\bar{n}]$.

\medskip
For convenience we often use (as in \cite{jojic_tverberg_2021}) an extended $\tau = (A_1,A_2; B)$ notation for simplices in the Bier sphere, where $B := [n]\setminus (A_1\cup A_2)$. Hence, an ordered partition $A_1\sqcup A_2 \sqcup B = [n]$ corresponds to a simplex $\tau\in Bier(K)$ if and only if $A_1 \in K$, $A_2 \in K^\circ$ (which implies $\emptyset \neq B \neq [n]$).

In the ``interval notation'', used in \cite{bjorner_bier_2004}, the simplex $\tau =  (A_1,A_2; B)$ is recorded as the pair $(A_1, [n]\setminus A_2)$ while the same simplex  is denoted in \cite[Section 5.6]{matousek_using_2008} by $\tau = A_1\uplus A_2$.

\medskip
For example, the facets of $Bier(K)$ are triples $\tau = (A_1, A_2; B)$ where $B = \{\nu\}$ is a singleton. In this case $\tau$ is
(in the interval notation) determined by the pair $A\subsetneq C$, where $A = A_1\in K$ and $C = A_1\cup \{\nu\}\notin K$, or in other words if $A$ is a losing and $C$ a wining coalition.

\begin{itemize}
\item $Fan(K) = Fan(\Gamma)$, the \emph{canonical fan}  of $K$ (respectively $\Gamma$), is a complete, simplicial fan in $H_0 := \{x\in \mathbb{R}^{n} \mid x_1+\dots+ x_n = 0\}$, associated to a simplicial complex $K\subsetneq 2^{[n]}$ (simple game $\Gamma$), introduced in \cite{jtz-bier-2019, jevtic_bier_2022}.
\end{itemize}
The rays ($1$-dimensional cones) of $Fan(K)$ are generated by vectors $\{\delta_1,\dots, \delta_n, \bar{\delta}_1,\dots, \bar{\delta}_n\}$, where $\delta_i = e_i-\frac{1}{n}(e_1+\dots+ e_n)$ and $\bar{\delta}_j := -\delta_j$. It was shown in \cite{jtz-bier-2019, jevtic_bier_2022} that the correspondence $i \mapsto \delta_i, j\mapsto \bar{\delta}_j$ defines a starshaped embedding of $Bier(K)$ into $H_0$ (centered at the origin), with $Fan(K)$ as the corresponding \emph{radial fan}. In particular
\[
      Cone(\tau) = Cone(A_1, A_2; B) := ConvCone\{\delta_i, \bar{\delta}_j \}_{i,j=1}^n
\]
is a cone in $Fan(K)$ if and only if $A_1 \in K$, $A_2 \in K^\circ$ and $A_1\cap A_2 = \emptyset$.

\section{Proof of Theorem \ref{thm:canonical}}\label{sec:canonical}

Our main tool for testing polytopality of a Bier sphere $Bier(K)$ is the method of \emph{wall crossing inequalities} \cite{albertin_removahedral_2020}. This method allows us to decide if a Bier sphere is polytopal with  the corresponding radial fan prescribed in advance and a natural choice for such a fan is the canonical fan $Fan(K)$.
Note however that the method does not provide a characterization. More precisely, if the sphere does not admit a polytopal sphere realization with prescribed radial fan, it might have such a realization with some other radial fan.

Theorem \ref{thm:canonical} says that such a realization is possible with the fan $Fan(K)$ if and only if $K$ is a threshold complex (i.e.\ if $\Gamma = 2^{[n]}\setminus K$ is a weighted majority game).

\subsection{Wall crossing inequalities}
\label{sec:wall-crossing}

The following proposition plays a central role in applications of the method of \emph{wall crossing inequalities}.

\begin{prop}\label{prop:wall-crossing}{\rm (\cite{albertin_removahedral_2020})}
Let $\mathcal{F}$ be an essential complete simplicial fan in
$\mathbb{R}^n$ and $\mathbf{G}$ be the $N\times n$ matrix whose rows are
the rays of $\mathcal{F}$. Then the following are equivalent for any
vector $\mathbf{h} \in \mathbb{R}^N$. \begin{enumerate}[label=(\arabic*)]
     \item[{\rm (1)}] The fan $\mathcal{F}$ is the normal fan of the polytope
$P_{\mathbf{h}}:=\{x \in \mathbb{R}^n \mid \mathbf{G}x \leqslant
\mathbf{h} \}$.
     \item[{\rm (2)}] For any two adjacent chambers $\mathbb{R}_{\geqslant
0}\mathbf{R}$ and $\mathbb{R}_{\geqslant 0}\mathbf{S}$ of $\mathcal{F}$
with $\mathbf{R}\setminus \{r\}=\mathbf{S}\setminus \{s\}$,
     \begin{align}\label{eqn:wall-inequality}
         \alpha\mathbf{h_r}+\beta\mathbf{h_s}+\sum_{\mathbf{t}\in
\mathbf{R}\cap \mathbf{S}} \gamma_{\mathbf{t}}\mathbf{h_t}>0,
     \end{align} where \begin{align}\label{eqn:wall-equality}
         \alpha\mathbf{r}+\beta\mathbf{s}+\sum_{\mathbf{t}\in
\mathbf{R}\cap \mathbf{S}} \gamma_{\mathbf{t}}\mathbf{t}=0
     \end{align}
     is the unique (up to scaling) linear dependence with $\alpha,\beta>0$
between the rays of $\mathbf{R} \cup \mathbf{S}$. \end{enumerate}
\end{prop}

\begin{cor}\label{cor:vazno}
The system of inequalities (\ref{eqn:wall-inequality}) (indexed by pairs of adjacent chambers) has a solution $\mathbf{h'}$ if and only it has a strictly positive solution $\mathbf{h}>0$.
\end{cor}
Indeed, if the system of inequalities (\ref{eqn:wall-inequality}) has a solution $\mathbf{h'}$ then $\mathcal{F}$ is the normal fan of the polytope
$P_{\mathbf{h'}}:=\{x \in \mathbb{R}^n \mid \mathbf{G}x \leqslant
\mathbf{h'} \}$, where all inequalities are essential (corresponding to facets). If $a$ is in the interior of $P_{\mathbf{h'}}$ then the polytope $P_{\mathbf{h'}}-a = P_{\mathbf{h}}$ has the same normal fan as the polytope  $P_{\mathbf{h'}}$ while $\mathbf{h}$ is positive.  \hfill $\square$

\medskip
We apply Proposition \ref{prop:wall-crossing} to the canonical fan $Fan(K)$ (described in Section \ref{sec:glossary}) which arises from the starshaped embedding of the Bier sphere $Bier(K)$, introduced in   \cite{jtz-bier-2019, jevtic_bier_2022}. The main result is the \emph{$K$-submodularity theorem} (Theorem \ref{thm:K-submodular}), see also  \cite[Theorem 19]{jevtic_bier_2022}. As emphasized in \cite{jevtic_bier_2022}, $K$-submodular functions play the role for polytopal Bier spheres similar to the role of classical submodular functions (polymatroids) in the theory of generalized permutohedra, see \cite{castillo_liu}.

\begin{defin}
Given a (proper) simplicial complex $K\subsetneq 2^{[n]}$,  an element $A\in K$ is a \emph{boundary simplex} if\ $(\exists c\in [n])\, A\cup\{c\}\notin K$. Similarly $B\notin K$ is a \emph{boundary non-simplex} if \ $(\exists c\in [n])\, B\setminus \{c\}\in K$. A pair $(A, B')\in (K,2^{[n]}\setminus K)$ is a \emph{boundary pair} if $B' = A\cup\{c\}$ for some $c\in [n]$.
\end{defin}
We already know  that boundary pairs $(A,B')$ correspond to maximal simplices in $Bier(K)$. In the following proposition we describe the \emph{ridges}, as the codimension one simplices in the Bier sphere $Bier(K)$ which are in bijective correspondence with the pairs of adjacent chambers in $Fan(K)$.

\begin{prop}\label{prop:ridges}
The ridges (codimension one simplices) $\tau \in Bier(K)$  have one of the following three forms, exhibited in Figure \ref{ex-3}. Here we use the \emph{interval notation} $\tau = (X,Y)$ (Section \ref{sec:glossary}) where $X\subsetneq Y, X\in K, Y\notin K$ and $(X,Y) \neq (\emptyset, [n])$.
\end{prop}

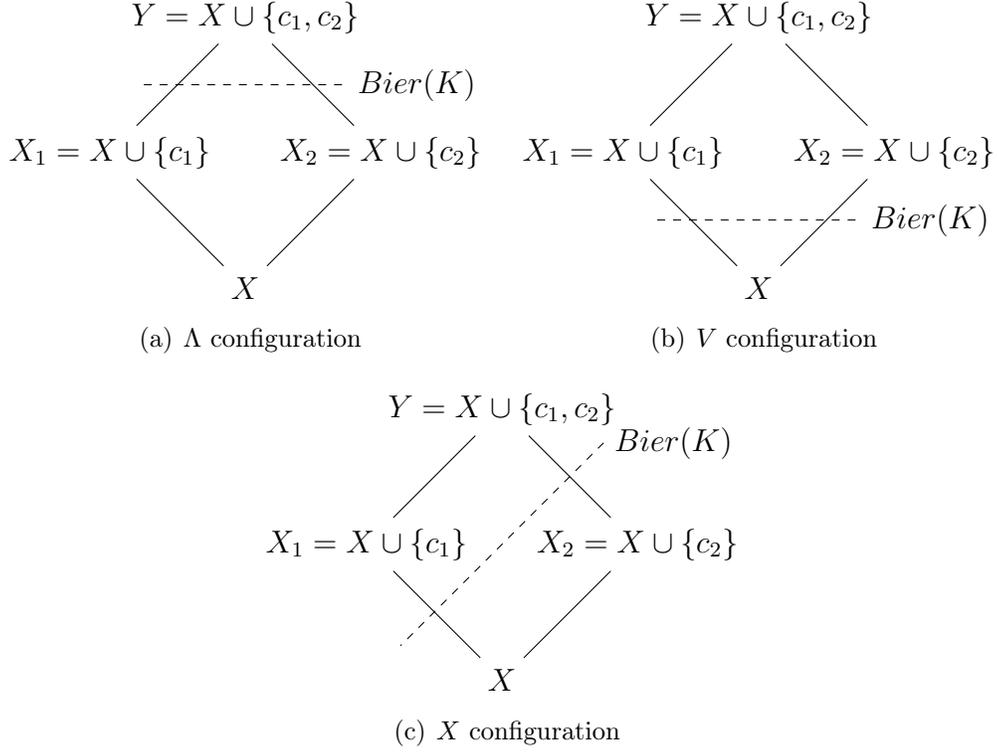
\begin{figure}[htb]
    \centering
    \subfigure[$\Lambda$ configuration]{\begin{tikzpicture}[scale=0.9]
    \node (y) at (0,2) {$Y=X\cup \{c_1,c_2\}$};
    \node (x1) at (-2,0) {$X_1=X\cup \{c_1\}$};
    \node (x2) at (2,0) {$X_2=X\cup \{c_2\}$};
    \node (x) at (0,-2) {$X$};

    \draw (x) -- (x1) -- (y);
    \draw (x) -- (x2) -- (y);

    \draw[dashed] (-1.5,1) -- (1.5,1);
    \node[anchor=west] (bier) at (1.5,1) {$Bier(K)$};
\end{tikzpicture}}
    \subfigure[$V$ configuration]{\begin{tikzpicture}[scale=0.9]
    \node (y) at (0,2) {$Y=X\cup \{c_1,c_2\}$};
    \node (x1) at (-2,0) {$X_1=X\cup \{c_1\}$};
    \node (x2) at (2,0) {$X_2=X\cup \{c_2\}$};
    \node (x) at (0,-2) {$X$};

    \draw (x) -- (x1) -- (y);
    \draw (x) -- (x2) -- (y);

    \draw[dashed] (-1.5,-1) -- (1.5,-1);
    \node[anchor=west] (bier) at (1.5,-1) {$Bier(K)$};
\end{tikzpicture}}
    \subfigure[$X$ configuration]{\begin{tikzpicture}[scale=0.9]
    \node (y) at (0,2) {$Y=X\cup \{c_1,c_2\}$};
    \node (x1) at (-2,0) {$X_1=X\cup \{c_1\}$};
    \node (x2) at (2,0) {$X_2=X\cup \{c_2\}$};
    \node (x) at (0,-2) {$X$};

    \draw (x) -- (x1) -- (y);
    \draw (x) -- (x2) -- (y);

    \draw[dashed] (1.5,1.5) -- (-1.5,-1.5);
    \node[anchor=west] (bier) at (1.5,1.5) {$Bier(K)$};
\end{tikzpicture}}
    \caption{Configurations of maximal adjacent simplices in $Bier(K)$.}
    \label{ex-3}
\end{figure}

\begin{proof} In the interval notation, the ridges in $Bier(K)$ correspond to intervals $(X,Y)$ where $Y= X\cup \{c_1, c_2\}$ and $c_1\neq c_2$. The $\Lambda$-configurations correspond to the case when both $X_1$ and $X_2$ are in $K$, the $V$-configurations  correspond to the case when neither $X_1$ nor $X_2$ are in $K$, and the $X$-configurations arise if precisely one of these sets is in $K$.
\end{proof}

\begin{defin}{\rm (\cite{jevtic_bier_2022})}\label{def:K-submodular}
Let $K\subsetneq 2^{[n]}$ be a simplicial complex and $Bier(K)$ the associated Bier sphere. A \emph{$K$-submodular function} ($K$-wall crossing function) is a function $$f : Vert(Bier(K))\rightarrow \mathbb{R}$$ such that
\begin{align}
f(c_1) + f(c_2) + \Sigma_{i\in X} f(i) > \Sigma_{j\notin Y} f(\bar{j}) & \mbox{\quad {\rm for each $\Lambda$-configuration} } \label{eq:K-1}\\ f(\bar{c}_1) + f(\bar{c}_2) + \Sigma_{j\notin Y} f(\bar{j}) > \Sigma_{i\in X} f(i) & \mbox{\quad {\rm for each $V$-configuration} } \label{eq:K-2}\\
f(c_2) + f(\bar{c}_2) > 0 & \mbox{\quad {\rm for each $X$-configuration}.} \label{eq:K-3}
\end{align}
\end{defin}

\begin{theo}{\rm (\cite[Theorem 19]{jevtic_bier_2022})}\label{thm:K-submodular}
  Let $\mathcal{F}= Fan(K)$ be the canonical fan (Section \ref{sec:glossary}), i.e.\ the radial fan arising from the canonical starshaped realization of the associated Bier sphere $Bier(K)$.  Then $\mathcal{F}$ is a normal fan of a convex polytope if and only if the simplicial complex $K$ admits a $K$-submodular function. Moreover, there is a bijection between convex realizations of $Bier(K)$ with radial fan $\mathcal{F}$ and $K$-submodular functions $f$.
\end{theo}

\begin{rem}\label{rem:feasibility}{\em
  Recall that in light of Corollary \ref{cor:vazno}, the condition (\ref{eq:K-3}) can be replaced by a stronger constraint  $f(c_2)>0$ and $f(\bar{c}_2) > 0$, without affecting the consistency (feasibility) of the whole system. }
\end{rem}
 \subsection{The ridges and the corresponding inequalities}\label{sec:listing}

 In order to list all inequalities (\ref{eq:K-1})-(\ref{eq:K-2})-(\ref{eq:K-3}) we need to make a list of all elements of $RidgeBier(K)$ and then classify them by their type
 \[
 RidgeBier(K) = RidgeBier_\Lambda(K) \cup RidgeBier_V(K) \cup RidgeBier_X(K) \, .
 \]
 For this purpose the following, more symmetric, notation is more convenient then the ``interval notation'', used in the previous section.

 \begin{prop}\label{prop:ridges-2}
   A ridge $\tau \in RidgeBier(K)$ can be described as the pair (triple)  $\tau = (X_1, X_2) = (X_1, \{c_1, c_2\}, X_2)$ where
   \[
       X_1 \uplus \{c_1, c_2\} \uplus X_2 = [n]
   \]
 is a partition, such that $X_1\in K$  and $X_2\in K^\circ$.
 \end{prop}
 In this notation the inequalities  (\ref{eq:K-1})-(\ref{eq:K-2})-(\ref{eq:K-3}) can be rewritten as follows:
 \begin{align}
x_{c_1} + x_{c_2} + \Sigma_{i\in X_1}\, x_i > \Sigma_{j\in X_2}\, y_j  & \mbox{\quad {\rm if $X_1\cup \{c_\nu\}\in K$}\, $(\nu=1,2)$ }  \label{eq:KK-1}\\
y_{{c_1}} + y_{{c_2}} + \Sigma_{j\in X_2}\, y_j > \Sigma_{i\in X_1}\, x_i &  \mbox{\quad {\rm if $X_2\cup \{c_\nu\}\in K^\circ$} $(\nu=1,2)$ }  \label{eq:KK-2}\\
x_{c_2} + y_{{c_2}} > 0 &\mbox{\quad {\rm if $X_1\cup \{c_2\}\in K$}. } \label{eq:KK-3}
\end{align}
 Note that the $X$-configuration condition in (\ref{eq:KK-3}) is symmetric in the sense that if $X_1\cup \{c_2\}\in K$ and $X_1\cup \{c_1\}\notin K$ then $X_2\cup \{c_2\}\in K^\circ$ and $X_2\cup \{c_1\}\notin K$, and vice versa.


\begin{prop}\label{prop:equivalence}
  If a simplicial complex $K\subset 2^{[n]}$ without ghost vertices admits a $K$-submodular function (Definition \ref{def:K-submodular}) then $K$ is a threshold complex.
\end{prop}

\proof  In light of Proposition \ref{prop:ridges-2} and inequalities (\ref{eq:KK-1})-(\ref{eq:KK-2})-(\ref{eq:KK-3}), $K$ admits a $K$-submodular function if there exist two vectors $x = (x_1, \dots, x_n)$ and $y = (y_1,\dots, y_n)$ in $\mathbb{R}^n$ which satisfy the inequalities (\ref{eq:KK-1}),  (\ref{eq:KK-2}) and (\ref{eq:KK-3}).

\begin{lema}
 The system of inequalities (\ref{eq:KK-1})-(\ref{eq:KK-2})-(\ref{eq:KK-3}) has a solution if and only if the following system also has a solution:
\begin{align}
x(S) > y(S^c)    & \mbox{\quad {\rm if $S\notin K$ }}  \label{eq:KKK-1}\\
x(T) < y(T^c)    & \mbox{\quad {\rm if $T\in K$ }}  \label{eq:KKK-2}\\
x> 0 \quad  y > 0 &       \label{eq:KKK-3}
\end{align}
where $x(S) = \sum_{i\in S} x_i$ etc., and $S^c = [n]\setminus S$ is the complement of $S$ in $[n]$.
\end{lema}
\noindent
Indeed, the inequalities (\ref{eq:KK-1})-(\ref{eq:KK-2}) are special cases of the inequalities (\ref{eq:KKK-1})-(\ref{eq:KKK-2}) while, by Remark \ref{rem:feasibility},  (\ref{eq:KK-3}) can be replaced by (\ref{eq:KKK-3}) without changing the consistency of the whole system.

\medskip
Conversely, let us show that, for illustration, (\ref{eq:KKK-1}) follows from (\ref{eq:KK-1}). Indeed, let $S_1 \subseteq S$ be of smallest cardinality such that $S_1\notin K$.
We observe that $S_1 \supseteq\{c_1,c_2\}$ must have at least two elements (otherwise $K$ would have ghost vertices). In turn $S_1 = S_2 \cup \{c_1,c_2\}$ describes a ridge of the type $\Lambda$ and the implication $(\ref{eq:KK-1}) \Rightarrow (\ref{eq:KKK-1})$ follows from
\[
      x(S) \geq x(S_1) > y(S_1^c) \geq y(S^c) \, .
\]
(Note that here we tacitly used (\ref{eq:KKK-3}) instead of (\ref{eq:KK-3}).)
The implications $(\ref{eq:KK-1}) \Rightarrow (\ref{eq:KKK-1})$ and  $(\ref{eq:KK-2}) \Rightarrow (\ref{eq:KKK-2})$ are proved analogously, which completes the proof of the lemma.    \hfill$\square$

\medskip\noindent
\emph{Proof of Proposition \ref{prop:equivalence} (cont.)}
Let $z = x+y$. Since
\[
               y(S^c) = y([n]) - y(S)  \qquad {and} \qquad  y(T^c) = y([n]) - y(T)
\]
the inequalities (\ref{eq:KKK-1}),  (\ref{eq:KKK-2}) and (\ref{eq:KKK-3}) can be rewritten as follows
\begin{align}
z(S) > y([n])    & \mbox{\quad {\rm if $S\notin K$ }}  \label{eq:KKKK-1}\\
z(T) < y([n])    & \mbox{\quad {\rm if $T\in K$ }}  \label{eq:KKKK-2}\\
z > 0 &       \label{eq:KKKK-3}
\end{align}
From here we conclude that $K = Tr_{z<y([n])}$ is the threshold complex with the quota $y([n])$ and weight distribution $z$. \hfill$\square$

\bigskip\noindent
{\bf Proof of Theorem \ref{thm:canonical}:} Proposition \ref{prop:equivalence} covers the first half of the statement of Theorem \ref{thm:canonical}. The second half is covered by the proof of \cite[Corollary 20]{jevtic_bier_2022}.  \hfill $\square$

\section{Roughly weighted simple games}

The class of \emph{roughly weighted simple games} \cite{taylor_games_2000, gvozdeva_games_2011} is considerably larger then the class of weighted majority games.
Recall that both classes can be characterized in terms of ``trading transforms'', by the results of Elgot (1961), and Taylor and Zwicker (1992), for the weighted majority games, and Gvozdeva and Slinko (2011), for the roughly weighted games.

Our definition is slightly more restrictive, compared to \cite{taylor_games_2000, gvozdeva_games_2011}, in the sense that we use only strictly positive weights (in order to avoid the appearance of ghost vertices in $K$).

\begin{defin}\label{def:rough-game}
  A simple game $(P,W)$, where $K = 2^P\setminus W$ is the collection of losing coalitions, is called \emph{roughly weighted} if there exist strictly positive real numbers $w = (w_1,\dots, w_n)$ and a positive real number $q$ (called the quota) such that for each $X\in 2^{P}$
  \begin{align}
w(X) = \sum_{i\in X} w_i > q  & \quad \Rightarrow \quad X\in W \label{eq:R-2}\\
w(X) = \sum_{i\in X} w_i < q  & \quad \Rightarrow \quad X\in K \label{eq:R-1}
\end{align}
\end{defin}
As a consequence of Definition \ref{def:rough-game} we obtain the implications
\[
  w(Z) < 1-q  \, \Rightarrow \, Z\in K^\circ \quad \mbox{ \rm and } \quad Z\in K^\circ \, \Rightarrow \, w(Z)\leq 1-q
\]
which say that if $(P,W)$ is roughly weighted, $2^P \setminus K^\circ$ is also a roughly weighted game.

\medskip
Note that by \ref{def:rough-game}, the set $\{X\in 2^P \mid w(X) = q\}$ is a antichain (clutter) in the sense that if $w(X_1) = w(X_2) = q$ then neither $X_1\varsubsetneq  X_2$ nor $X_1\varsupsetneq X_2$.

\medskip
We approach the proof of Theorem \ref{thm:rough-canonical} by formulating non-strict analogues of the key definitions and propositions from Section \ref{sec:canonical}.

\begin{defin}\label{def:LKK-submodular}
Let $K\subsetneq 2^{[n]}$ be a simplicial complex and $Bier(K)$ the associated Bier sphere. A \emph{non-strict $K$-submodular function} is a function $$f : Vert(Bier(K))\rightarrow \mathbb{R}$$ such that
\begin{align}
f(c_1) + f(c_2) + \Sigma_{i\in X} f(i) \geqslant \Sigma_{j\notin Y} f(\bar{j}) & \mbox{\quad {\rm for each $\Lambda$-configuration} } \label{eq:LK-1}\\ f(\bar{c}_1) + f(\bar{c}_2) + \Sigma_{j\notin Y} f(\bar{j}) \geqslant \Sigma_{i\in X} f(i) & \mbox{\quad {\rm for each $V$-configuration} } \label{eq:LK-2}\\
f(c_2)> 0 \quad f(\bar{c}_2) > 0 & \mbox{\quad {\rm for each $X$-configuration}.} \label{eq:LK-3}
\end{align}
\end{defin}

\begin{prop}\label{prop:rough-1} Suppose that $(P,W)$ is a roughly weighted game and let $K=2^P\setminus W$ be the associated set of losing coalitions. Then  $Bier(K)$ admits a non-strict $K$-submodular function.
\end{prop}

\begin{proof} Put $P=[n]$ and assume that $(P,W)$ is a roughly weighted game with parameters $[q; w_1,\dots, w_n]$.  Without loss of generality we assume that $w([n])=1$ and $0<q<1$.

Let us show that the function  $f : [n]\cup [\bar{n}]\rightarrow \mathbb{R}$, where $[n]\cup [\bar{n}] = Vert(Bier(K)) = \{1,\dots, n, \bar{1}, \dots, \bar{n}\}$, defined by
\begin{equation}\label{eq:Lf-threshold}
  f(i) = (1-q)w_i \qquad f(\bar{j}) = q w_j \qquad (i,j = 1,\dots, n)
\end{equation}
is indeed non-strictly $K$-submodular for $Bier(K)$. The inequalities \eqref{eq:LK-1} and \eqref{eq:LK-2}, for the function $f$ defined by \eqref{eq:Lf-threshold}, take  the following form
\begin{equation}\label{eq:both}
  (1-q)w(Y) \geqslant qw(Y^c)  \qquad (1-q)w(X)\leqslant qw(X^c) \, .
\end{equation}
However, in a roughly weighted game both inequalities  \eqref{eq:both} hold without any restrictions on a simplex $X\in K$ and a non-simplex $Y\notin K$. For example the second inequality in \eqref{eq:both} is a consequence of $w(X)\leqslant q$ and $w(X^c) \geqslant 1-q$.
\end{proof}

\begin{prop}\label{prop:rough-2}  Suppose that  $Bier(K)$ admits a non-strict $K$-submodular function.   Then $(P,W)$ is a roughly weighted game where $W=2^P\setminus K$ is the associated set of winning coalitions.
\end{prop}
\begin{proof}
The proof is analogous to the proof of Proposition \ref{prop:equivalence} with the main difference that the strict inequalities in (\ref{eq:KK-1})-(\ref{eq:KK-2})-(\ref{eq:KK-3}) (and elsewhere) are replaced by the inequalities (\ref{eq:LK-1})-(\ref{eq:LK-2})-(\ref{eq:LK-3}). As a consequence we obtain the inequalities:
\begin{align}
z(S) \geqslant y([n])    & \mbox{\quad {\rm if $S\notin K$ }}  \label{eq:LKKKK-1}\\
z(T) \leqslant y([n])    & \mbox{\quad {\rm if $T\in K$ }}  \label{eq:LKKKK-2}\\
z > 0 &       \label{eq:LKKKK-3}
\end{align}
We conclude that $(P,W)$ is a roughly weighted game with quota $y([n])$ and weights $z$. Indeed,  (\ref{eq:R-1}) is precisely the counterposition of the implication (\ref{eq:LKKKK-1}) while (\ref{eq:R-2}) is the counterposition of the implication (\ref{eq:LKKKK-2}).
\end{proof}

\subsection{Proof of Theorem \ref{thm:rough-canonical}}

Propositions \ref{prop:rough-1} and \ref{prop:rough-2} together imply that $(P,W)$ is a roughly weighted game with $K=2^P\setminus W$ as the associated set of losing coalitions if and only if $Bier(K)$ admits a non-strict $K$-submodular function. In light of this equivalence Theorem \ref{thm:rough-canonical} is a consequence of the following proposition.

\begin{prop}\label{prop:deformation}
Let $K\subsetneq 2^{[n]}$ be a proper simplicial complex without ghost vertices. Then $Bier(K)$ admits a non-strict $K$-submodular function (Definition \ref{def:LKK-submodular}) if and only if the canonical fan $Fan(\Gamma)$ of $\Gamma$ is \emph{pseudo-polytopal} in the sense that it refines the normal fan of a convex polytope.
\end{prop}

This proposition is a consequence of the following non-strict relative of Proposition \ref{prop:wall-crossing}.

\begin{prop}\label{prop:wall-crossing-non-strict}
Let $\mathcal{F}$ be an essential complete simplicial fan in
$\mathbb{R}^n$ and let $\Delta$ be a set of vectors representing $1$-dimensional cones in $\mathcal{F}$.  Let $\{\mathbf{{h}_t}\}_{\mathbf{t}\in \Delta}$ be a collection of strictly positive real numbers and let $F : V \rightarrow \mathbb{R}$ be a continuous function which is linear on each maximal cone of $\mathcal{F}$  such that $F(\mathbf{t}) = \mathbf{h_t}$ for each $t\in\Delta$.    Then the following are equivalent:
 \begin{enumerate}[label=(\arabic*)]

     \item[{\rm (0)}] The function $F$ is convex.
     \item[{\rm (1)}] The function $F$ is  the support function of a (unique) convex polytope $Q$ in $V$, such that the fan  $\mathcal{F}$ is a refinement of the normal fan $\mathcal{N}(Q)$.
     \item[{\rm (2)}] For any two adjacent chambers $\mathbb{R}_{\geqslant
0}\mathbf{R}$ and $\mathbb{R}_{\geqslant 0}\mathbf{S}$ of $\mathcal{F}$
with $\mathbf{R}\setminus \{r\}=\mathbf{S}\setminus \{s\}$,
     \begin{align}\label{eqn:wall-inequality-2}
         \alpha\mathbf{h_r}+\beta\mathbf{h_s}+\sum_{\mathbf{t}\in
\mathbf{R}\cap \mathbf{S}} \gamma_{\mathbf{t}}\mathbf{h_t}\geqslant0,
     \end{align} where \begin{align}\label{eqn:wall-equality-2}
         \alpha\mathbf{r}+\beta\mathbf{s}+\sum_{\mathbf{t}\in
\mathbf{R}\cap \mathbf{S}} \gamma_{\mathbf{t}}\mathbf{t}=0
     \end{align}
     is  (up to scaling) the unique linear dependence with $\alpha,\beta>0$
between the vectors of $\mathbf{R} \cup \mathbf{S}$. \end{enumerate}
\end{prop}

\begin{proof}
The implication ${\rm (1)} \Rightarrow {\rm (0)}$ is trivial. The implication ${\rm (0)} \Rightarrow {\rm (1)}$ follows from the observation that if the support function $F = h_Q$ of $Q$ is linear (and positive) on each maximal cone of $\mathcal{F}$, then $\mathcal{F}$ is a refinement of the normal fan $\mathcal{N}(Q)$.

 The implication ${\rm (0)} \Rightarrow {\rm (2)}$ follows from the observation that (\ref{eqn:wall-inequality-2}) is a consequence of convexity of $F$. Conversely, ${\rm (2)} \Rightarrow {\rm (0)}$ follows, mutatis mutandis, by the argument of the proof of Lemma 2.1 in \cite{chapoton_associahedron_2002}.
\end{proof}


\section{Proof of Theorem \ref{thm:Bier-10-vertices}}

Theorems \ref{thm:canonical} and \ref{thm:rough-canonical} provide a complete characterization of (roughly) weighted games (threshold complexes) in terms of the canonical polytopality (pseudo-polytopality) of the corresponding Bier spheres.

It is known \cite[Section 5.6]{matousek_using_2008} that with the increase of the number of vertices (number of players) Bier spheres tend to be nonpolytopal. In other words nonpolytopal Bier spheres certainly exist, but we do not know how far we should go to find them (Problem \ref{prob:unsolved})

Theorem \ref{thm:Bier-10-vertices}, and the corresponding algorithm for proving polytopality of Bier spheres, show that nonpolytopal Bier spheres must have at least 12  vertices. In particular the ``M\" obius Bier sphere'' (Section \ref{sec:revisited}) is polytopal, although (in light of Theorem \ref{thm:canonical}) it is not canonically polytopal, being a Bier sphere of a non-threshold complex.

\subsection{Polytpal Bier spheres of non-threshold complexes}

We describe an algorithm which tries to find an explicit polytopal realization of a given Bier sphere $Bier(K)$ by a sequence of modifications, where the initial step is the canonical polytopal realization of the Bier sphere $Bier(L)$ of a threshold complex $L$ (chosen to be as close to $K$ as possible).\medskip

\begin{figurehere}
\centering 
\includegraphics{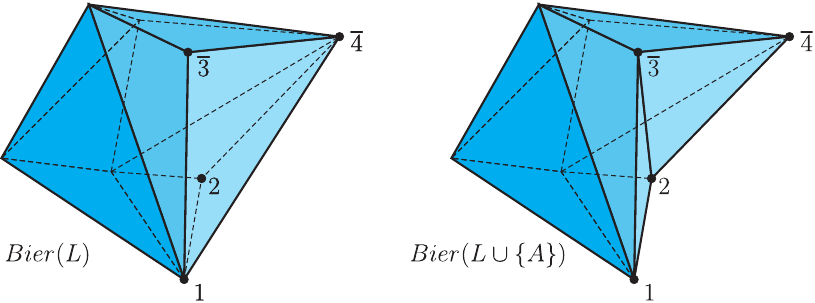}
 \caption{Re-triangulation of a disk on a Bier sphere.}
 \label{Figura:Bistelarna}
\end{figurehere}\medskip

Let $L\subseteq 2^{[n]}$ be a simplicial complex and $A\subseteq [n]$ such that $K=L\cup\{A\}$ is also a simplicial complex. Then
$$Bier(L\cup\{A\})=Bier(L)\setminus\big\{(A\setminus \{i\})\uplus A^c\mid i\in A\big\}\cup\big\{A\uplus (A\cup \{i\})^c\mid i\not\in A\big\}.$$
In other words $Bier(K)$ is obtained form $Bier(L)$ by a re-triangulation of a disk with vertices  $A\uplus A^c$, as illustrated by Figure \ref{Figura:Bistelarna}. A more detailed description of this operation can be found in \cite{matousek_using_2008}.\medskip

If $Bier(L)$ is isomorphic to the border of a convex polytope $\conv \{x_1,\ldots,x_n,y_1\ldots,y_n\}$ where vertices $1,\ldots,n$ of $Bier(L)$ correspond to $x_1,\ldots,x_n$ and vertices $\overline{1},\ldots,\overline{n}$ correspond to $y_1,\ldots,y_n$, one can try to obtain a geometrical realization of $Bier(K)$ by varying the points $x_i$, $i\in A$ and $y_i$, $i\in A^{c}$. Simplifying the calculation, we shall  vary the points $v\in \{x_1,\ldots,x_n,y_1\ldots,y_n\}$ radially along the lines $Ov$ and where $O$ is the origin. \medskip

A simplicial complex $\Sigma\subseteq 2^{[m]}$ is a triangulation of the border of a $d-$dimensional convex polytope $\conv \{v_1,\ldots,v_m\}$ if and only if for every maximal simplex (facet) $A$ of $\Sigma$, the points $v_j$, $j\not \in A$ are on the same side of a plane spanned by the points $v_i$, $i\in A$. If $\vec{A}=(i_1,\ldots,i_d)$ is a chosen orientation of the simplex $A$, then simplices $\vec{A}*j=(i_1,\ldots,i_k,j)$ have the same orientation for every $j\not\in A$. This can be verified geometrically in a following way, if $v_{i}=\{v_{i1},\ldots,v_{id}\}\in \mathbb{R}^d$ for $i=1,\ldots,m$ than, the orientation of a geometrical simplex $v_{\vec{A}*j}=(v_{i_1},\ldots,v_{i_{d}},v_j)$ is the sign of a following determinant:
$$[v_{\vec{A}*j}]=\left|
                    \begin{array}{cccc}
                      v_{i_1 1} & \cdots & v_{i_1 d} & 1 \\
                      \vdots    &        & \vdots    & \vdots  \\
                      v_{i_d 1} & \cdots & v_{i_d d} & 1 \\
                      v_{j_1 1} & \cdots & v_{j_d 1} & 1 \\
                    \end{array}
                  \right|
$$
Since the sphere is  orientable, we may assume that orientations of the maximal simplices $A\in Max(\Sigma)$ are chosen so that all determinants $[v_{\vec{A}*j}]$ have the same sign equal to one.\medskip

Let $v=\{v_1,\ldots,v_m\}$ be arbitrary configuration of points in an Euclidean space $\mathbb{R}^{d}$ and let $tv=\{t_1 v_1,\ldots,t_m v_m\}$ be a configuration obtained form $v$ by radially varying points $v_i$ trough multiplication by scalars $t_i \in \mathbb{R}$. If $\Sigma\subseteq 2^{[m]}$ is a triangulation of a $d-$dimensional sphere where the maximal simplices  $A\in Max(\Sigma)$ are suitably oriented then a convex geometrical realization of $\Sigma$ can be obtained trough radial variation of the vertexes $v$ iff the system of inequalities
\begin{equation}\label{eq:SISTEM}
 [tv_{\vec{A}*j}]>0,\quad A\in Max(\Sigma),\quad j\in A^c
 \end{equation}
has a solution.\medskip

The system (\ref{eq:SISTEM}) is non-linear, has $m$ real variables $t_1,\ldots,t_m$ and its solvability largely depends on starting configuration of points so, for the method to be efficient, one needs to start with points  which are sufficiently close to targeted geometrical realization.

\bigskip

Let $\mu:[n]\rightarrow \mathbb{R}^{+}$ be a probability measure given by a vector $(\mu_1,\ldots,\mu_n)$ and  $\alpha\in\mathbb{R}^{+}$ such that $\mu(A)\neq \alpha$ for all $A\subseteq [n]$ and let $L=Tr_{\mu<\alpha}=\{A\subseteq [n]\mid \mu(A)<\alpha\}$ be the associated threshold complex.
In \cite{jtz-bier-2019} was shown that $Bier(L)$ is isomorphic to the border of a convex polytope spanned by the points $x_1,\ldots,x_n,y_1,\ldots,y_n\in\mathbb{R}^{n-1}$ given by\medskip
\begin{equation}\label{eq:REALIZACIJA}
x_i=\frac{1}{\mu_i}b_i,\quad y_i=-\frac{1-\alpha}{\alpha}b_i,\quad i=1,\ldots,n
\end{equation}
where $b_1,\ldots,b_n\in\mathbb{R}^{n-1}$ is a ``minimal circuit'' i.e. a set of points such that each of its proper subset is linearly independent and $b_1+\cdots+b_n=0$. In this setting, vertices $1,\ldots,n$ of $L\subset Bier(L)$ are mapped to $x_1,\ldots,x_n$ and vertices $\overline{1},\ldots,\overline{n}$ of $L^{\circ}\subset Bier(L)$ are mapped to $y_1,\ldots,y_n$.\medskip

Let $K\subseteq 2^{[n]}$ be a non-threshold simplicial complex. To obtain a geometrical realization of $Bier(K)$, we start with geometrical configuration of a Bier sphere $Bier(L)$ where $L$ is a maximal subcomplex of $K$ which is threshold. To determine $L$, one can use the threshold characteristic of $K$ introduced in \cite{zivaljevic_unavoidable_2020}. Then,
$$K=L\cup\{A_1,\ldots,A_k\}$$
and if simplices  $A_1,\ldots,A_k$ are ordered increasingly by dimension, the simplicial complexes $K_0=L$, $K_i=K_{i-1}\cup\{A_i\}$, $i=1,\ldots,k$ form a sequence
$$L=K_{0}\subset K_1\subset K_2\subset \cdots \subset K_k=K$$
which corresponds to the sequence of Bier spheres
$$Bier(L)=Bier(K_0), Bier(K_1),\ldots, Bier(K_k)=Bier(K)$$
where each sphere $Bier(K_i)$ is obtained from $Bier(K_{i-1})$ by re-triangulating a disk with vertices  $A_i\uplus A_{i}^c$.\medskip

If $\conv v$ where $v=\{x_1,\ldots,x_n,y_1\ldots,y_n\}$ is a convex realization of $Bier(K_{i-1})$ then, to obtain a convex realization of $Bier(K_{i})$, we introduce a variable configuration $$tv=\{t_1 x_1,\ldots,t_n x_n,s_1 y_1\ldots,s_n y_n\}$$
where
\begin{align}
\nonumber t_i &\in \mathbb{R},\ i\in A,\quad \ \ t_i=1,\ i\in A^c,\\
\nonumber s_j &\in \mathbb{R},\  j\in A^c,\quad  s_j=1,\ j\in A,
\end{align}
and solve the system of inequalities (\ref{eq:SISTEM}) which now has $n$ variables.\medskip

\medskip
Let us demonstrate the method on one simple example. Let $K=2^{\{1,2\}}\cup 2^{\{3,4\}}$ be the simplicial complex whose geometric realization is the disjoint union of two segments. This simplicial complex is not threshold, it's maximal subcomplex which is threshold is $L=K\setminus \{A\}$ for $A=\{1,2\}$ where the required probability measure is given by the vector $\mu=(\frac{1}{3},\frac{1}{3},\frac{1}{6},\frac{1}{6})$ and threshold $\alpha=\frac{5}{12}$. If we choose a minimal circuit
$$b_1=(1,0,0),\ b_2=(0,1,0),\ b_3=(0,1,0),\ b_4=(-1,-1,-1) \, ,$$
by (\ref{eq:REALIZACIJA}) the sphere $Bier(L)$ is the boundary of the  convex polytope whose  vertices are
\begin{align}
\nonumber v=\bigg\{&(3,0,0),(0,3,0),(0,0,6),(-6,-6,-6),\\
\nonumber         &\left(-\frac{21}{5},0,0\right),\left(0,-\frac{21}{5},0\right),\left(0,0,-\frac{42}{5}\right),\left(\frac{42}{5},\frac{42}{5},\frac{42}{5}\right)\bigg\}
\end{align}
which is shown in  Figure \ref{Figura:Bistelarna} on the left. To construct geometrical realization of $Bier(K)$, we first  translate all points of $v$ for the vector $(3,3,0)$, since we want to perform radial variation (modification) of vertices
$1,2,\overline{3},\overline{4}$ in approximately the same direction, as indicated in  Figure \ref{Figura:Radijalno} on the left.\medskip

The associated  ``variable point configuration'' is the following
\begin{align}
\nonumber tv=\bigg\{& \left(6 t_1,3 t_1,0\right),\left(3 t_2,6 t_2,0\right),(3,3,6),(-3,-3,-6),\\
\nonumber           &\left(-\frac{6}{5},3,0\right),\left(3,-\frac{6}{5},0\right),\left(3 s_3,3 s_3,-\frac{42 s_3}{5} \right),\left(\frac{57 s_4}{5},\frac{57 s_4}{5},\frac{42 s_4}{5}\right)\bigg\}
\end{align}
and the system (\ref{eq:SISTEM}), for convex realization of the Bier sphere $Bier(K)$, obtains the following form:
\begin{align}
\nonumber 5 t_1-1>0,\ 5 t_2-1>0,\ 5 t_2 t_1+3 t_1-4 t_2>0,\ 5 t_2 t_1-4 t_1+3 t_2>0,\ 47 s_3-5>0,\\
\nonumber s_3 t_1-15 t_1+28 s_3>0,\ 12 s_3 t_1-5 t_1+7 s_3>0,\ 41 s_3 t_1+15 t_1-28 s_3>0,\\
\nonumber s_3 t_2-15 t_2+28 s_3>0,\ 12 s_3 t_2-5 t_2+7 s_3>0,\ 41 s_3 t_2+15 t_2-28 s_3>0,\\
\nonumber 5 t_1 t_2+7 t_1 s_3 t_2-4 s_3 t_2-4 t_1 s_3>0,\ -5 t_1 t_2+7 t_1 s_3 t_2+4 s_3 t_2+4 t_1 s_3>0,\ 47 s_4-5>0,\\
\nonumber s_4 t_1-15 t_1+28 s_4>0,\ 12 s_4 t_1-5 t_1+7 s_4>0,\ 41 s_4 t_1+15 t_1-28 s_4>0,\\
\nonumber s_4 t_2-15 t_2+28 s_4>0,\ 12 s_4 t_2-5 t_2+7 s_4>0,\ 41 s_4 t_2+15 t_2-28 s_4>0,\\
\nonumber 5 t_1 t_2+7 t_1 s_4 t_2-4 s_4 t_2-4 t_1 s_4>0,\ -5 t_1 t_2+7 t_1 s_4 t_2+4 s_4 t_2+4 t_1 s_4>0,\\
\nonumber 14 s_4 s_3+5 s_3-5 s_4>0,\ 14 s_4 s_3-5 s_3+5 s_4>0,\\
\nonumber 15 t_1 s_3+40 t_1 s_4 s_3-56 s_4 s_3+15 t_1 s_4>0,\ 5 t_1 t_2 s_3-8 t_1 s_4 s_3-8 t_2 s_4 s_3+5 t_1 t_2 s_4>0,\\
\nonumber 15 t_2 s_3+40 t_2 s_4 s_3-56 s_4 s_3+15 t_2 s_4>0 \, .
\end{align}

This system has a solution
$$(t_1,t_2,s_3,s_4)=\left(\frac{51}{25},\frac{44}{25},\frac{21}{20},\frac{131}{100}\right)$$
which, when inserted  in the variable point configuration,  provides a convex realization of $Bier(K)$, depicted in  Figure \ref{Figura:Radijalno} (on the right).

\begin{figure}
\centering 
\includegraphics{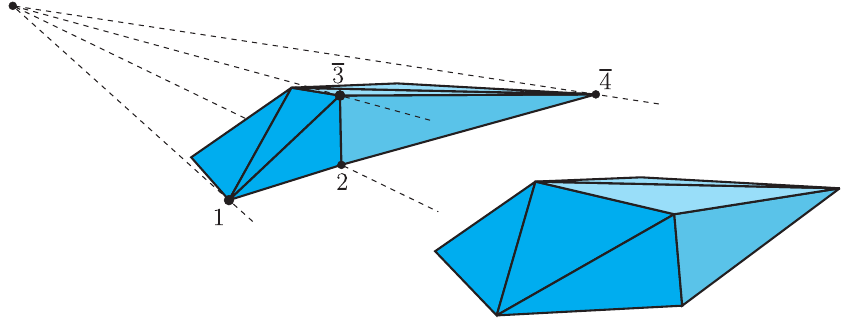}
 \caption{Radial variation of vertices.}
 \label{Figura:Radijalno}
\end{figure}\bigskip

For our second example, let us construct a convex geometric realization of the Bier sphere of the pentagonal cycle $K\subset 2^{[5]}$,  depicted in Figure \ref{Figura:Mebijus} on the left. The complex $K$ is not a threshold complex and  its Alexander dual is the minimal triangulation of the M\"{o}bius band, depicted in Figure \ref{Figura:Mebijus} on the right. For this reason $Bier(K)$ is sometimes referred to as the M\"{o}bius Bier sphere.

\medskip

We start with a complex $L=2^{[2]}\cup 2^{\{2,3\}}\cup \big\{\{4\},\{5\}\big\}$ which is a threshold subcomplex of a complex $K$ for the measure $\mu=\left\{\frac{3}{10},\frac{1}{50},\frac{1}{25},\frac{8}{25},\frac{8}{25}\right\}$ and the  threshold $\alpha=\frac{33}{100}$. If we choose a minimal circuit
$$b_1=(1,0,0,0),\quad b_2=(0,1,0,0),\quad b_3=(0,0,1,0),\quad b_4=(0,0,0,1), \quad b_5=-(1,1,1,1)$$ then by  (\ref{eq:REALIZACIJA})  the sphere $Bier(L)$ has a convex realization with vertices listed as the  rows of the following matrix

$$V_{Bier(L)}=\left[
\begin{array}{cccc}
 \frac{10}{3} & 0 & 0 & 0 \\
 0 & 50 & 0 & 0 \\
 0 & 0 & 25 & 0 \\
 0 & 0 & 0 & \frac{25}{8} \\
 -\frac{25}{8} & -\frac{25}{8} & -\frac{25}{8} & -\frac{25}{8} \\
 -\frac{670}{99} & 0 & 0 & 0 \\
 0 & -\frac{3350}{33} & 0 & 0 \\
 0 & 0 & -\frac{1675}{33} & 0 \\
 0 & 0 & 0 & -\frac{1675}{264} \\
 \frac{1675}{264} & \frac{1675}{264} & \frac{1675}{264} & \frac{1675}{264} \\
\end{array}
\right]$$

\begin{figurehere}
\centering 
\includegraphics{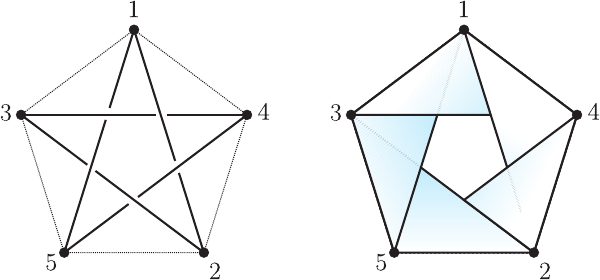}
 \caption{Pentagonal cycle and its dual, the M\"{o}bius band.}
 \label{Figura:Mebijus}
\end{figurehere}

\bigskip

Let $K_1=L\cup \big\{\{1,5\}\big\}$, $K_2=K_1\cup \big\{\{3,4\}\big\}$ and $K_3=K_2\cup \big\{\{4,5\}\big\}=K$ which  leads to the following sequence of successive approximations
$$L\subset K_1\subset K_2\subset K.$$

By applying the general method, and combining successive re-triangulations of disks on intermediate  Bier spheres with  radial variations of their vertices, and using numerical methods for solving inequalities (\ref{eq:SISTEM}), we finally obtain convex realizations of Bier spheres $Bier(K_1)$, $Bier(K_2)$ and $Bier(K)$. The matrices of their convex  realizations are:
$$V_{Bier(K_1)}=\left[
\begin{array}{cccc}
 3.37502 & -5.09311 & -2.85102 & -0.896848 \\
 0.0296025 & 50.7462 & 0.155815 & -0.368407 \\
 0.0296025 & 0.746166 & 25.1558 & -0.368407 \\
 0.0296025 & 0.746166 & 0.155815 & 2.75659 \\
 -3.07544 & -12.7875 & -8.32932 & -4.43607 \\
 -6.73807 & 0.746166 & 0.155815 & -0.368407 \\
 -0.0659424 & -52.6647 & 24.8672 & 3.97454 \\
 0.0087548 & 11.2174 & -45.1973 & 0.579216 \\
 0.0325764 & -0.747559 & -0.613352 & -6.84851 \\
 6.3743 & 7.09086 & 6.50051 & 5.97629 \\
\end{array}
\right]; $$
$$
V_{Bier(K_2)}=\left[
\begin{array}{cccc}
 3.38521 & -3.86036 & -4.47002 & -1.2378 \\
 0.0397926 & 51.9789 & -1.46319 & -0.709358 \\
 -0.0109639 & -3.9592 & 31.3401 & 0.930843 \\
 0.0118228 & -1.29333 & 2.82878 & 3.3205 \\
 -3.06525 & -11.5548 & -9.94833 & -4.77702 \\
 -6.71696 & 3.16451 & -3.01825 & -1.03684 \\
 -0.0557521 & -51.432 & 23.2482 & 3.63359 \\
 0.0189449 & 12.4502 & -46.8163 & 0.238265 \\
 0.0427665 & 0.485194 & -2.23236 & -7.18946 \\
 6.35039 & 4.02084 & 10.5314 & 6.82728 \\
\end{array}
\right];
$$
\begin{equation}\label{eq:poly-0}
V_{Bier(K_3)}=\left[
\begin{array}{cccc}
 2.58418 & -7.84043 & -7.2259 & -1.12488 \\
 -0.761242 & 47.9988 & -4.21908 & -0.59644 \\
 -0.811999 & -7.93927 & 28.5842 & 1.04376 \\
 -0.538235 & -4.02784 & 0.93516 & 3.39776 \\
 -2.77768 & -10.1347 & -8.96312 & -4.81456 \\
 -7.23817 & 0.568376 & -4.81495 & -0.963001 \\
 2.40876 & -39.1324 & 31.681 & 3.28216 \\
 2.3433 & 23.9615 & -38.7646 & -0.0884603 \\
 -0.758268 & -3.49488 & -4.98824 & -7.07654 \\
 5.54935 & 0.0407717 & 7.77551 & 6.9402 \\
\end{array}
\right]
\end{equation}

\medskip

By implementing the algorithm in some of the standard programming languages 
and using numerical methods for solving the system (\ref{eq:SISTEM}), many convex geometrical realizations of higher dimensional Bier spheres of non threshold complexes where obtained.\medskip

\medskip
It turns out that in the case of $3$-dimensional Bier spheres on 10 vertices, there are $88$ non-threshold complexes on $5$ vertices and  $48$ non-isomorphic Bier spheres associated to them. Using described method, we where able to construct convex realizations of all these spheres which  led to the conclusion (Theorem \ref{thm:Bier-10-vertices}) that there are no nonpolytopal $3$-dimensional Bier spheres on $10$ vertices. The pdf file with the list of all complexes and coordinates of convex realizations of the corresponding Bier spheres is available for download at the address  \cite{kv5d3}.\medskip

\medskip
In case of $4$-dimensional Bier spheres on 11 vertices, all $88$ non-threshold complexes on $5$-vertices (viewed in the ambient $[6]$) have non-isomorphic Bier spheres and all of them are polytopal. Consequently, there are no nonpolytopal $4$-dimensional Bier spheres on $11$ vertices. The complexes and coordinates of convex realizations of their Bier spheres are given in \cite{kv5d4}.\medskip

As for the $4$-dimensional spheres on $12$ vertices, the algorithm was able to construct convex realizations of many of these spheres, but not all of them. The downside of the method is that there is significantly more non-threshold then threshold complexes with high number of vertices, so there is an insufficient number of suitable triangulations to initiate the algorithm.

\section{Bier sphere of the $5$-cycle revisited}\label{sec:revisited}

In this section we provide an independent verification that   $V_{Bier(K_3)}$ is indeed a correct convex realization of the $10$-vertex Bier sphere $Bier(C_5)$ associated to a $5$-cycle $C_5$. For this purpose we use the program {\emph{Polymake} \cite{polymake} and this may serve as an illustration how in principle all convex realizations listed in \cite{kv5d3, kv5d4}, can be verified.

The case of $Bier(C_5)$ is particularly interesting since, at an earlier stage of the project, there was some indication that it might be the desired nonpolytopal Bier sphere (Problem \ref{prob:unsolved}). As it turned out it is quite the opposite, by Theorem \ref{thm:Bier-10-vertices} all Bier spheres with 10 vertices are polytopal,

It is well-known that the $(C_5)^\circ$, Alexander dual  of $C_5$, is the minimal (5-vertex) triangulation of the M\" obius band, depicted in Figure \ref{fig:penta-1} on the left.

}
\begin{figure}[htb]
\centering
\includegraphics[scale=0.75]{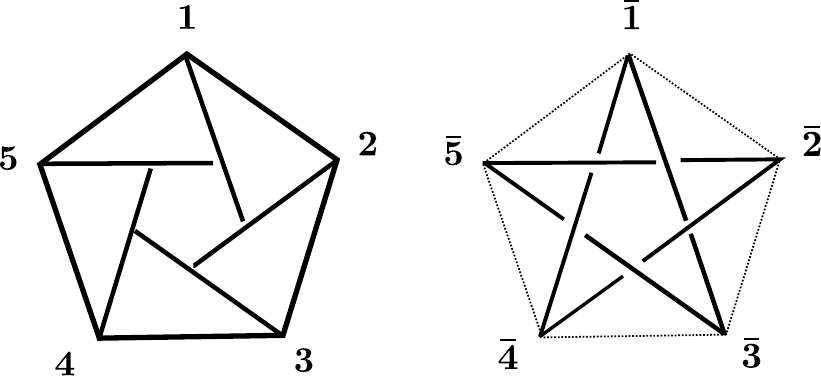}
\caption{Triangulated M\" obius band as a dual of a $5$-cycle.}
\label{fig:penta-1}
\end{figure}

Maximal simplices of $Bier(C_5)$ are four element subsets $A\cup \overbar{B}\subset [5]\cup [\bar{5}]$, where $A
\in (C_5)^\circ, B\in C_0$ and $A$ and $B$ are disjoint (as subsets of [5]).

\begin{figure}[htb]
\centering
\includegraphics[scale=0.75]{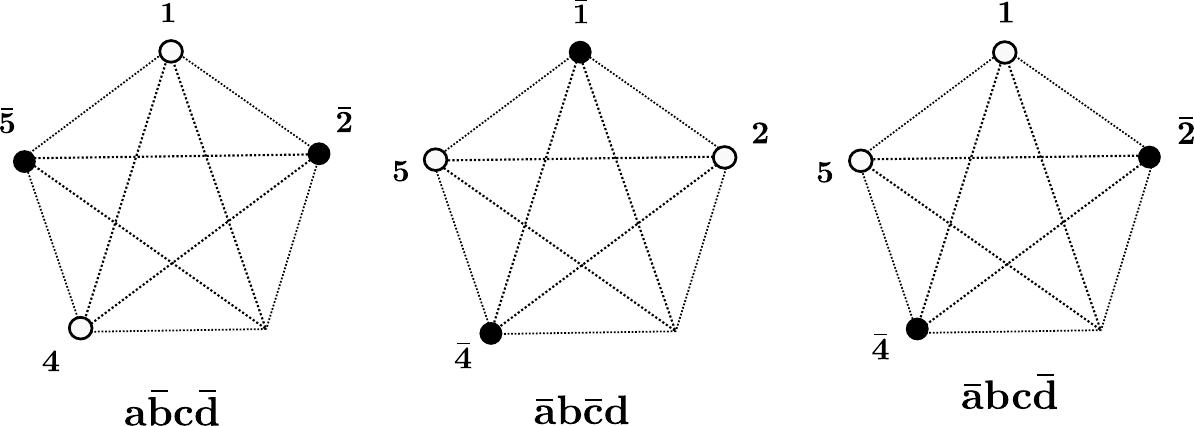}
\caption{Simplices of the type $(2,2)$.}
\label{fig:Tip-(2,2)}
\end{figure}

\begin{figure}[htb]
\centering
\includegraphics[scale=0.75]{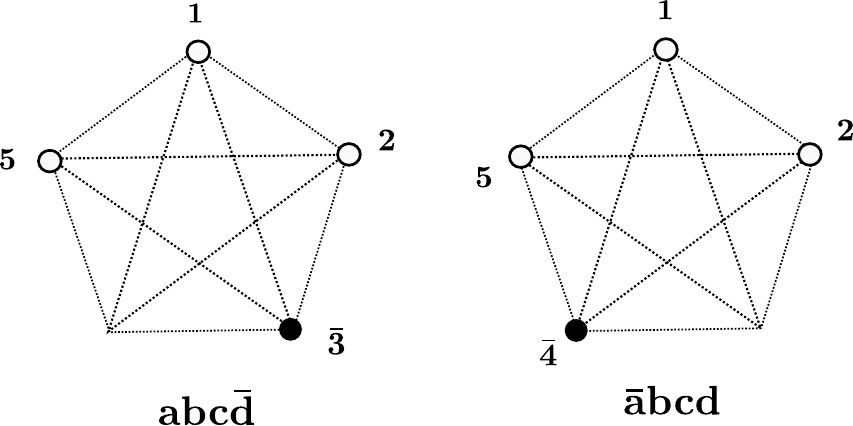}
\caption{Simplices of the type $(3,1)$.}
\label{fig:Tip-(3,1)}
\end{figure}

\subsection{The fundamental homology class of $Bier(C_5)$}

Here we list in the table (\ref{fund-class}) all maximal simplices (tetrahedra) of the Bier sphere $Bier(C_5)$. The simplices are  properly (coherently) oriented, so that their formal sum defines a fundamental homology class (cycle) in $H_3(Bier(C_5); \mathbb{Z})$.

 As illustrated in Figures \ref{fig:Tip-(2,2)} and \ref{fig:Tip-(3,1)}, the order of vertices is always clockwise, starting with the first vertex  after the ``missing vertex'' of the pentagon. The sign of each (ordered) simplex is indicated at the end of the corresponding row of (\ref{fund-class}).

 \begin{equation}\label{fund-class}
\begin{array}{llllll}
1\overbar{2}3\overbar{4} & \overbar{2}3\overbar{4}5 & 3\overbar{4}5\overbar{1} & \overbar{4}5\overbar{1}2 & 5\overbar{1}2\overbar{3} & + \\
\overbar{1}2\overbar{3}4 & 2\overbar{3}4\overbar{5} & \overbar{3}4\overbar{5}1 & 4\overbar{5}1\overbar{2} & \overbar{5}1\overbar{2}3 & +\\
\overbar{1}23\overbar{4} & \overbar{2}34\overbar{5} & \overbar{3}45\overbar{1} & \overbar{4}51\overbar{2} & \overbar{5}12\overbar{3} & +\\
123\overbar{4} & 234\overbar{5} & 345\overbar{1} & 451\overbar{2} & 512\overbar{3} &  - \\
\overbar{1}234 & \overbar{2}345 & \overbar{3}451 & \overbar{4}512 & \overbar{5}123  & - \, .
\end{array}
\end{equation}
The calculation of the signs is straightforward, however, for the reader's convenience, we include some details. By convention our boundary operator read as follows
\[
   \partial (a_1,a_2,\dots, a_n) = \sum_{i=1}^{n} (-1)^i (a_1,\dots, a_{i-1}, a_{i+1}, \dots, a_n) \, .
\]
Since $\partial [1\overbar{2}3\overbar{4}] = -[\overbar{2}3\overbar{4}] \pm \dots$ and $\partial [\overbar{2}3\overbar{4}5] = +[\overbar{2}3\overbar{4}] \pm \dots$, we conclude that $[1\overbar{2}3\overbar{4}]$ and $[\overbar{2}3\overbar{4}5]$ must be of the same sign, which we take to be positive. This argument accounts for the sign of all simplices in the first two rows of (\ref{fund-class}). Similarly, since $\partial [1\overbar{2}3\overbar{4}] = +[13\overbar{4}] \pm \dots$ and $\partial [123\overbar{4}] = +[13\overbar{4}] \pm \dots$ we observe that $[1\overbar{2}3\overbar{4}]$ and $[123\overbar{4}]$ have the opposite sign in the fundamental class of $Bier(C_5)$. This in turn determines the sign of all simplices  the last two rows of (\ref{fund-class}). Finally, since
\[
      \partial [123\overbar{4}] = - [23\overbar{4}]\pm\dots  \quad \mbox{and} \quad  \partial [\overbar{1}23\overbar{4}] = - [23\overbar{4}]\pm\dots
\]
we obtain the sign of all simplices in the third row of (\ref{fund-class}).

\subsection{Face lattice of the polytope (\ref{eq:poly-0}) }

Let $Q = Conv(V)$ be the four dimensional convex polytope, obtained as the convex hull of ten points (labeled as $0,1,\dots, 8, 9)$, which are listed as the rows of the matrix (\ref{eq:poly-0}).

\medskip
\emph{Polymake}, applied to the matrix (\ref{eq:poly-0}), produces as an output the face lattice of the polytope $Q$. All we need here is the list of all facets, recorded in the table (\ref{eq:poly-1}), and the list (\ref{eq:poly-2})  of all edges of the polytope $Q$.

\begin{equation}\label{eq:poly-1}
Facets(Q)=\left\lbrace
\begin{array}{ccccc}
{1 2 5 8} & {1 5 7 8} & {0 1 7 8} & {4 5 7 8} & {2 3 5 6} \\
{3 4 6 7} & {2 5 6 8} & {4 5 6 8} & {3 4 5 6} & {0 4 6 8} \\
 {0 4 7 8} & {0 4 6 7} & {3 4 5 7} & {3 5 7 9} & {0 6 7 9} \\
{0 6 8 9} & {2 6 8 9} & {3 6 7 9} & {2 3 5 9} & {2 3 6 9} \\
{1 2 8 9} & {1 2 5 9} & {1 5 7 9} & {0 1 8 9} & {0 1 7 9}
\end{array}
\right\rbrace
\end{equation}

\begin{equation}\label{eq:poly-2}
Edges(Q)=\left\lbrace
\begin{matrix}
{0 1} & {0 4} & {0 6} & {0 7} & {0 8} & {0 9} & {1 2} \\
{1 5} & {1 7} & {1 8} & {1 9} & {2 3} & {2 5} & {2 6} \\
{2 8} & {2 9} & {3 4} & {3 5} & {3 6} & {3 7} & {3 9} \\
{4 5} & {4 6} & {4 7} & {4 8} & {5 6} & {5 7} & {5 8} \\
{5 9} & {6 7} & {6 8} & {6 9} & {7 8} & {7 9} & {8 9}
\end{matrix}
\right\rbrace
\end{equation}

\medskip
We need to convince ourselves that the combinatorics of tables (\ref{fund-class}) and (\ref{eq:poly-1})  is essentially the same, i.e.\ that there exists a bijection
\[
 \sigma : \{1,2,3,4,5, \bar{1}, \bar{2}, \bar{3}, \bar{4}, \bar{5}\} \longrightarrow \{0, 1, 2, 3, 4, 5, 6, 7, 8, 9\}
\]
such that $abcd$ is listed in (\ref{fund-class}) if and only if $\sigma(a)\sigma(b)\sigma(c)\sigma(d)$ appears on the list (\ref{eq:poly-1}).

\medskip
In order to narrow the search for $\sigma$ we observe that each symbol (vertex) appears in the corresponding list either 8 or 12 times. For example the vertices $\{\bar{i}\}_{i=1}^5$ of the cycle $C_5$, depicted in Figure \ref{fig:penta-1} on the right, are precisely the elements that appear in the list (\ref{fund-class}) exactly 8 times, while the corresponding elements in (\ref{eq:poly-1}) are $\{0,1,2,3,4\}$.

\medskip
Moreover,  $\{01, 12, 23, 34, 04\}$ are the only pairs from $\{0,1,2,3,4\}$ which appear in (\ref{eq:poly-2}) while the edges of $C_5$ are $\{\bar{1}\bar{3}, \bar{3}\bar{5}, \bar{5}\bar{2}, \bar{2}\bar{4}, \bar{4}\bar{1}\}$. By taking into account the symmetries of $Bier(C_5)$, it is natural to define $\sigma^{-1}(i) = \overline{2i+1}\, (\mbox{mod } 5)$ for $0\leq i\leq 4$, or equivalently $\sigma(k) = 3k+2 \, (\mbox{mod } 5)$ for $1\leq k\leq 5$.

\medskip
There is natural bijection $i \leftrightarrow \bar{i}$ on the set of vertices of $B(C_5)$. The corresponding bijection on the vertices of $Q$ is $i \leftrightarrow i +5 \, (i=0,1,2,3,4)$.

\medskip
Indeed,  $i\bar{i}$ is never an edge in $B(C_5)$ and if $i\in \{0,1,2,3,4\}, j\in \{5,6,7,8,9\}$ then $ij$ is listed in (\ref{eq:poly-2}) if and only if $j\neq i+5$.
It follows that for $j\in \{5,6,7,8,9\}$ one has $\sigma^{-1}(j) = 2(j-5)+1 \, (\mbox{mod } 5)$.

\medskip
Finally, by inspection, we check that $\sigma$ is indeed an isomorphism of lists (\ref{fund-class}) and (\ref{eq:poly-1}).

\vspace{2cc}
\noindent
\textbf{Acknowledgments.} The second and the third author were supported by the Science Fund of the Republic of Serbia, Grant No.\ 7744592, Integrability and Extremal Problems in Mechanics, Geometry and Combinatorics - MEGIC.



\begin{thebibliography}{10}

\bibitem{albertin_removahedral_2020}
D.~Albertin, V.~Pilaud, and J.~Ritter.
\newblock Removahedral congruences versus permutree congruences, 2022.

\bibitem{bjorner_bier_2004}
A.~Björner, A.~Paffenholz, J.~Sjöstrand, and G.~M. Ziegler.
\newblock Bier spheres and posets.
\newblock {\em Discrete and Computational Geometry}, 34(1):71--86, sep 2004.

\bibitem{castillo_liu}
F.~Castillo and F.~Liu.
\newblock {Deformation Cones of Nested Braid Fans}.
\newblock {\em International Mathematics Research Notices}, 2022(3):1973--2026,
  06 2020.

\bibitem{chapoton_associahedron_2002}
F.~Chapoton, S.~Fomin, and A.~Zelevinsky.
\newblock Polytopal realizations of generalized associahedra.
\newblock {\em Canadian Mathematical Bulletin}, 45(4):537--566, dec 2002.

\bibitem{cukic-blowups-2007}
S.~L. {\v{C}}uki{\'{c}} and E.~Delucchi.
\newblock Simplicial shellable spheres via combinatorial blowups.
\newblock {\em Proceedings of the American Mathematical Society},
  135(08):2403--2415, apr 2007.

\bibitem{longueville_bier_2004}
M.~de~Longueville.
\newblock Bier spheres and barycentric subdivision.
\newblock {\em Journal of Combinatorial Theory, Series A}, 105(2):355--357, feb
  2004.

\bibitem{delucchi_fundamental_2020}
E.~Delucchi and L.~Hoessly.
\newblock Fundamental polytopes of metric trees via parallel connections of
  matroids.
\newblock {\em European Journal of Combinatorics}, 87:103098, jun 2020.

\bibitem{ewald_geometry_1996}
G.~Ewald.
\newblock {\em Combinatorial Convexity and Algebraic Geometry}.
\newblock Springer New York, 1996.

\bibitem{polymake}
E.~Gawrilow and M.~Joswig.
\newblock polymake: a framework for analyzing convex polytopes.
\newblock In {\em Polytopes {\textemdash} Combinatorics and Computation}, pages
  43--73. Birkhäuser Basel, 2000.

\bibitem{gvozdeva_games_2011}
T.~Gvozdeva, L.~A. Hemaspaandra, and A.~Slinko.
\newblock Three hierarchies of simple games parameterized by
  {\textquotedblleft}resource{\textquotedblright} parameters.
\newblock {\em International Journal of Game Theory}, 42(1):1--17, nov 2011.

\bibitem{jtz-bier-2019}
F.~D. Jevti{\'{c}}, M.~Timotijevi{\'{c}}, and R.~T. {\v{Z}}ivaljevi{\'{c}}.
\newblock Polytopal {Bier} spheres and {Kantorovich--Rubinstein} polytopes of
  weighted cycles.
\newblock {\em Discrete {\&} Computational Geometry}, Nov 2019.

\bibitem{jevtic_bier_2022}
F.~D. Jevti\'{c} and \v{Z}ivaljevi\'{c} Rade~T.
\newblock Bier spheres of extremal volume and generalized permutohedra.
\newblock {\em Applicable Analysis and Discrete Mathematics}, 2022.

\bibitem{jojic_tverberg_2021}
D.~Jojić, G.~Panina, and R.~Živaljević.
\newblock A {Tverberg} type theorem for collectively unavoidable complexes.
\newblock {\em Israel J. Math.}, Jan. 2021.

\bibitem{lutz_combinatorial_2007}
F.~H. Lutz.
\newblock Combinatorial 3-manifolds with 10 vertices, 2007.

\bibitem{matousek_using_2008}
J.~Matoušek.
\newblock {\em Using the {Borsuk}–{Ulam} {Theorem}}.
\newblock Springer Berlin Heidelberg, Berlin, Heidelberg, 2008.

\bibitem{zivaljevic_unavoidable_2020}
M.~J. Milutinovi{\'{c}}, D.~Joji{\'{c}}, M.~Timotijevi{\'{c}}, S.~T.
  Vre{\'{c}}ica, and R.~T. {\v{Z}}ivaljevi{\'{c}}.
\newblock Combinatorics of unavoidable complexes.
\newblock {\em European Journal of Combinatorics}, 83:103004, jan 2020.

\bibitem{taylor_games_2000}
A.~D. Taylor and W.~S. Zwicker.
\newblock {\em Simple Games}.
\newblock Princeton University Press, 2000.

\bibitem{kv5d3}
M.~Timotijevi\'{c}.
\newblock
  \href{https://imi.pmf.kg.ac.rs/pub/m_timotijevic/bier_kv5_d3.pdf}{https://imi.pmf.kg.ac.rs/pub/m\_timotijevic/bier\_kv5\_d3.pdf}.

\bibitem{kv5d4}
M.~Timotijevi\'{c}.
\newblock
  \href{https://imi.pmf.kg.ac.rs/pub/m_timotijevic/bier_kv5_d4.pdf}{https://imi.pmf.kg.ac.rs/pub/m\_timotijevic/bier\_kv5\_d4.pdf}.

\bibitem{timotijevic_dual_2019}
M.~Timotijevi\'{c}.
\newblock Note on combinatorial structure of self-dual simplicial complexes.
\newblock {\em Matemati\v{c}ki Vesnik}, 1(71):104--122, 2019.

\end{thebibliography}

\bigskip
\noindent
\textbf{Marinko Timotijevi\'{c}} \\
University of Kragujevac, Faculty of Science\\ Radoja Domanovi\'{c}a 12,
Kragujevac, Serbia \\
E-mail: {\it timotijevicmarinko@yahoo.com}

\noindent
\textbf{Rade T. \v{Z}ivaljevi\'{c}}\\
Mathematical Institute,\\
Serbian Academy of Sciences and Arts,\\
Belgrade, Serbia,\\
E-mail: {\it rade@turing.mi.sanu.ac.rs}

\noindent
\textbf{Filip D. Jevti\'{c}} \\
Mathematical Institute, \\
Serbian Academy of Sciences and Arts,\\
Belgrade, Serbia,\\
E-mail: {\it filip@turing.mi.sanu.ac.rs}

\end{document}